\documentclass[11pt,a4paper,twoside]{amsart}

\usepackage{epsfig}
\usepackage{amsfonts}
\usepackage{amsmath}
\usepackage{amssymb}
\usepackage{amsthm}
\usepackage{enumerate}
\usepackage[T1]{fontenc}
\usepackage[latin9]{inputenc}
\usepackage{ae,aecompl}
\usepackage{graphicx}
\usepackage{psfrag}
\usepackage{nicefrac}
\usepackage{mathrsfs}
\usepackage{color}

\usepackage{hyperref}

\newtheorem{theorem}{Theorem}[section]
\newtheorem{corollary}[theorem]{Corollary}
\newtheorem{lemma}[theorem]{Lemma}
\newtheorem{proposition}[theorem]{Proposition}

{\theoremstyle{remark}
\newtheorem{remark}[theorem]{Remark}}
{\theoremstyle{definition}
\newtheorem{definition}[theorem]{Definition}

\newtheorem{notation}[theorem]{Notation}
\newtheorem{example}[theorem]{Example}

}

\newcommand{\PP}[0]{\ensuremath{\mathbb{P}}}

\newcommand{\ZZ}[0]{\ensuremath{\mathbb{Z}}}
\newcommand{\AF}[0]{\ensuremath{\mathbb{A}}}

\newcommand{\QQ}[0]{\ensuremath{\mathbb{Q}}}
\newcommand{\TT}[0]{\ensuremath{\mathbb{T}}}
\newcommand{\KK}[0]{\ensuremath{k}}
\newcommand{\OO}[0]{\ensuremath{\mathcal{O}}}
\newcommand{\tA}[0]{\ensuremath{\widetilde{A}}}
\newcommand{\tX}[0]{\ensuremath{\widetilde{X}}}
\newcommand{\DD}[0]{\ensuremath{\mathfrak{D}}}

\newcommand{\FF}[0]{\ensuremath{\mathscr{F}}}

\newcommand{\spec}[0]{\ensuremath{\operatorname{Spec}}}
\newcommand{\supp}[0]{\ensuremath{\operatorname{Supp}}}

\newcommand{\divi}[0]{\ensuremath{\operatorname{div}}}

\newcommand{\pol}[0]{\ensuremath{\operatorname{Pol}}}

\newcommand{\rank}[0]{\ensuremath{\operatorname{rank}}}

\newcommand{\homo}[0]{\ensuremath{\operatorname{Hom}}}
\newcommand{\relint}[0]{\ensuremath{\operatorname{rel.int}}}
\newcommand{\pic}[0]{\ensuremath{\operatorname{Pic}}}

\begin{document}


\title[]{Rational singularities of normal $\TT$-varieties}

\author{Alvaro Liendo}
\address{Universit\'e
Grenoble I, Institut Fourier, UMR 5582 CNRS-UJF, BP 74, 38402
St.\ Martin d'H\`eres c\'edex, France}
\email{alvaro.liendo@ujf-grenoble.fr}

\date{\today}

\thanks{
\mbox{\hspace{11pt}}{\it 2000 Mathematics Subject
Classification}:
14J17, 14E15.\\
\mbox{\hspace{11pt}}{\it Key words}:  torus actions, $\TT$-varieties, rational singularities, Cohen-Macaulay singularities, elliptic surface singularities.}

\begin{abstract}
A $\TT$-variety is an algebraic variety $X$ with an effective regular action of an algebraic torus $\TT$. Altmann and Hausen \cite{AlHa06} gave a combinatorial description of an affine $\TT$-variety $X$ by means of polyhedral divisors. In this paper we compute the higher direct images of the structure sheaf of a desingularization of $X$ in terms of this combinatorial data. As a consequence, we give a criterion as to when a $\TT$-variety has rational singularities. We also provide a partial criterion for a $\TT$-variety to be Cohen-Macaulay. As an application we characterize in this terms quasihomogeneous elliptic singularities of surfaces.
\end{abstract}

\maketitle



\section*{Introduction} 
Let $\KK$ be an algebraically closed field of characteristic 0. In this paper we study $\TT$-varieties i.e., varieties endowed with an action of an algebraic torus $\TT=(\KK^*)^n$. A $\TT$-action on $X=\spec\,A$ gives rise to an $M$-grading on $A$, where $M$ is a lattice of rank $n$ and vice versa, any effective $M$-grading appears in this way (see e.g. \cite{KaRu82}). The complexity of a $\TT$-action is the codimension of a general orbit. For an effective $\TT$-action, the complexity equals $\dim X-\dim \TT$. 

There are well known combinatorial descriptions of normal $\TT$-varieties. For toric varieties see e.g., \cite{Dem70}, \cite[Ch. 1]{KKMS73}, and \cite{Oda88}. For affine $\KK^*$-surfaces see \cite{FlZa03c}. For complexity 1 case see \cite[Ch. 2 and 4]{KKMS73}, and more generally \cite{Tim08}. Finally, for arbitrary complexity see \cite{AlHa06,AHS08}.

We let $N=\homo(M,\ZZ)$,  $M_{\QQ}=M\otimes\QQ$, and $N_{\QQ}=N\otimes\QQ$. Any affine toric variety can be described via the weight cone $\sigma^\vee\subseteq M_{\QQ}$ spanned over $\QQ_{\geq0}$ by all $m\in M$ such that $A_m\neq\{0\}$ or, alternatively, via the dual cone $\sigma\subseteq N_{\QQ}$. Similarly, the description of normal affine $\TT$-varieties due to Altmann and Hausen deals with a polyhedral cone $\sigma\subseteq N_{\QQ}$ (dual to the weight cone $\sigma^\vee\subseteq M_{\QQ}$), a normal semiprojective variety $Y$, and a divisor $\DD$ on $Y$ whose coefficients are polyhedra in $N_{\QQ}$ invariant by translations in $\sigma$.

Let $X$ be a normal variety and let $\psi:Z\rightarrow X$ be a desingularization. Usually, the classification of singularities involves the higher direct images of the structure sheaf $R^i\psi_*\OO_Z$. In particular, $X$ has rational singularities if $R^i\psi_*\OO_Z=0$ for all $i\geq 1$, see e.g., \cite{Art66,KKMS73,Elk78}. 

Rational singularities are Cohen-Macaulay. Recall that a local ring is Cohen-Macau\-lay if its Krull dimension equals to the depth. A variety $X$ is called Cohen-Macaulay if all the local rings $\OO_{X,x}$ are Cohen-Macaulay. By a well known theorem of Kempf \cite[p. 50]{KKMS73}, a variety $X$ has rational singularities if and only if $X$ is Cohen-Macaulay and the induced map $\psi_*:\omega_Z\hookrightarrow \omega_X$ is an isomorphism. 

The content of the paper is as follows. In Section 1 we recall the combinatorial description of $\TT$-varieties due to Altmann and Hausen. In Section 2  we obtain our main results concerning the higher direct images of the structure sheaf of a desingularization of a $\TT$-variety. More precisely, in Theorem \ref{Tdirect} we compute the higher direct image sheaves $R^i\psi_*\OO_Z$ for a $\TT$-variety $X$ in terms of the Altmann-Hausen combinatorial data $(Y,\DD,\sigma)$. In Theorem \ref{Trat} we give a criterion in this terms as to when a $\TT$-variety has rational singularities. In Corollary \ref{cmgen} we apply Kempf's result to give a condition for a $\TT$-variety to be Cohen-Macaulay.

Finally in Section 4 we apply previous results to characterize elliptic singularities of $\KK^*$-surfaces. A normal surface singularity $(X,x)$ is called elliptic if $R^i\psi_*\OO_Z=0$ for all $i\geq 2$ and $R^1\psi_*\OO_Z=\KK$, see  e.g., \cite{Lau77,Wat80,Yau80}. An elliptic singularity is called minimal if it is Gorenstein i.e., is Cohen-Macaulay and the canonical sheaf $\omega_X$ is invertible. In Proposition \ref{gor} we give a criterion for a surface with a 1-torus action to be Gorenstein. In Theorem \ref{ellip} we characterize  (minimal) elliptic singularities in the combinatorial terms.

\smallskip
The author is grateful to Mikhail Zaidenberg for posing the problem and permanent encouragement, and to Hubert Flenner for useful discussions.



\section{Preliminaries}

\subsection{Combinatorial description of affine $\TT$-varieties}
\label{AHdes}

In this section we recall briefly the combinatorial description of affine $\TT$-varieties given by Altmann and Hausen in \cite{AlHa06}.

Let $N$ be a lattice of rank $n$ and $M=\homo(N,\ZZ)$ be its dual lattice. We fix dual bases $\{\nu_1,\cdots,\nu_n\}$ and $\{\mu_1,\cdots,\mu_n\}$ for $N$ and $M$, respectively. We also let $N_{\QQ}=N\otimes\QQ$, $M_{\QQ}=M\otimes\QQ$, and we consider the natural duality $M_{\QQ}\times N_{\QQ}\rightarrow \QQ$, $(m,p)\mapsto \langle m,p\rangle$. 

Let $\TT=\spec\KK[M]$ be the $n$-dimensional algebraic torus associated to $M$ and
let $X=\spec\,A$ be an affine $\TT$-variety. It is well known that the morphism $A\rightarrow A\otimes \KK[M]$ induces an $M$-grading on $A$ and, conversely, every $M$-grading on $A$ arises in this way. Furthermore, a $\TT$-action is effective if an only if the corresponding $M$-grading is effective. 

Let $\sigma$ be a pointed polyhedral cone in $N_{\QQ}$. We denote the dual cone by $\sigma^\vee$ and we let $\sigma_M^\vee=\sigma^\vee\cap M$. We define $\pol_{\sigma}(N_{\QQ})$ to be the set of all polyhedra in $N_{\QQ}$ which can be decomposed as the Minkowski sum of a compact polyhedron and $\sigma$. The Minkowsi sum endows $\pol_{\sigma}(N_{\QQ})$ with a structure of semigroup with neutral element $\sigma$.

To any polyhedron $\Delta\in\pol_{\sigma}(N_{\QQ})$ we associate its support function $h_{\Delta}:\sigma^\vee\rightarrow \QQ$ defined by $h_{\Delta}(m)=\min\langle m,\Delta\rangle$. Clearly the function $h_{\Delta}$ is  piecewise linear. Furthermore, $h_{\Delta}$ is upper convex and positively homogeneous i.e., 
$$h_{\Delta}(m+m')\geq h_{\Delta}(m)+h_{\Delta}(m'),\ \mbox{and}\ h_{\Delta}(\lambda m)=\lambda h_{\Delta}(m),\forall m,m'\in \sigma^{\vee},\ \forall\lambda\in \QQ_{\geq 0}\,.$$ 

\begin{definition} \label{ppd}
A variety $Y$ is called semiprojective if it is projective over an affine variety. A \emph{$\sigma$-polyhedral divisor} on $Y$ is a formal finite sum $\DD=\sum_{D}\Delta_D\cdot D$, where $D$ runs over prime divisors on $Y$, $\Delta_D\in\pol_{\sigma}(N_{\QQ})$ and $\Delta_D=\sigma$ for all but finitely many $D$. For $m\in\sigma^{\vee}$ we can evaluate $\DD$ in $m$ by letting $\DD(m)$ be the $\QQ$-divisor
$$\DD(m)=\sum_{D} h_D(m) D\,,$$
where $h_D=h_{\Delta_D}$. A $\sigma$-polyhedral divisor $\DD$ is called \emph{proper} if the following hold.
\begin{enumerate}[(i)]
\item The evaluation $\DD(m)$ is a semiample\footnote{A $\QQ$-divisor $D$ is called semiample if there exists $r>1$ such that the linear system $|rD|$ is base point free.} $\QQ$-Cartier divisor for all $m\in\sigma_M^\vee$, and
\item $\DD(m)$ is big for all $m\in\relint(\sigma^\vee)$.
\end{enumerate}
\end{definition}

The following theorem gives a combinatorial description of $\TT$-varieties analogous to the classical combinatorial description of toric varieties.

\begin{theorem} \label{AH}
Let $\DD$ be a proper $\sigma$-polyhedral divisor on a semiprojective variety $Y$. Consider the 
schemes\footnote{For a $\QQ$-divisor $D$, $\OO(D)$ stands for $\OO(\lfloor D\rfloor)$, where $\lfloor D\rfloor$ is the integral part of $D$.}
$$\tX=\spec_Y \tA[Y,\DD], \quad \mbox{where} \quad \tA[Y,\DD]=\bigoplus_{m\in\sigma_M}\OO_Y(\DD(m))\,,$$ 
$$\mbox{and}\quad X=\spec A[Y,\DD], \quad \mbox{where} \quad A[Y,\DD]=H^0(\tX,\OO_{\tX})=H^0(Y,\tA)\,.$$
Then the following hold.
\begin{enumerate}[(i)]
\item 
The schemes $X$ and $\tX$ are normal varieties of dimension $\dim Y+\rank M$.
\item
The $M$-gradings on $A[C,\DD]$ and $\tA[C,\DD]$ induce effective $\TT$-actions on $X$ and $\tX$,  respectively. Moreover, the canonical morphism $\pi:\tX\rightarrow Y$ is a good quotient for the $\TT$-action on $\tX$.
\item
There is a $\TT$-equivariant proper contraction $\varphi:\tX\rightarrow X$.
\end{enumerate}

Conversely, any normal affine variety with an effective $\TT$-action is equivariantly isomorphic to some $X=\spec A[Y,\DD]$ for a quasiprojective variety $Y$ and a proper $\sigma$-polyhedral divisor $\DD$ on $Y$.
\end{theorem}

\begin{remark} \label{aff}
In the case where $Y$ is affine the contraction $\varphi$ is an isomorphism since any quasi-coherent sheaf on an affine variety is generated by global sections.
\end{remark}

\subsection{Divisors on $\TT$-varieties} \label{T-div}

In \cite{FlZa03c} a characterization of $\TT$-invariant divisors on an affine $\KK^*$-surface is given. In \cite{PeSu08} some of these results are generalized to the case of a $\TT$-variety of complexity 1. In this section we recall results from \cite{PeSu08} and add some minor generalizations that we need. 

Letting $\DD$ be a proper polyhedral divisor on a semiprojective variety $Y$ we let $X=\spec A[Y,\DD]$ and $\tX=\tA[Y,\DD]$. Since the contraction $\varphi:\tX\rightarrow X$ is equivariant, the $\TT$-invariant prime Weil divisors on $X$ are in bijection with the $\TT$-invariant prime Weil divisors on $\tX$ not contracted by $\varphi$.

We first apply the orbit decomposition of the variety $\tX$ in Proposition 7.10 and Corollary 7.11 of \cite{AlHa06} to obtain a description of the $\TT$-invariant prime Weil divisors in $\tX$.  There are 2 types of $\TT$-invariant prime Weil divisors on $\tX$: the horizontal type corresponding to families of $\TT$-orbits closures of dimension $\rank M-1$ over $Y$; and the vertical type corresponding to families of $\TT$-orbits closures of dimension $\rank M$ over a prime divisor on $Y$.

\begin{lemma}
Let $\DD=\sum_D\Delta_D\cdot D$ be a proper $\sigma$-polyhedral divisor on a normal semiprojective variety $Y$. Letting $\tX=\tA[Y,\DD]$, the following hold.
\begin{enumerate}[(i)]
 \item The $\TT$-invariant prime Weil divisors on $\tX$ of horizontal type are in bijection with the extremal rays of $\rho\subseteq\sigma$. 
 \item The $\TT$-invariant prime Weil divisors on $\tX$ of vertical type are in bijection with pairs $(D,p)$ where $D$ is a prime Weil divisor on $Y$ and $p$ is a vertex of $\Delta_D$. 
\end{enumerate}
\end{lemma}
\begin{proof}
The lemma follows from Proposition 7.10 and Corollary 7.11 of \cite{AlHa06}. See also the proof of Proposition 3.13 in \cite{PeSu08}.
\end{proof}

The following lemma is a reformulation of Proposition 3.13 in \cite{PeSu08}.
\begin{lemma} \label{Tcont}
Let $\DD=\sum_D\Delta_D\cdot D$ be a proper $\sigma$-polyhedral divisor on a normal semiprojective variety $Y$. The following hold.
\begin{enumerate}[(i)]
\item Let $\rho\subseteq\sigma$ be an extremal ray and let $\tau$ be is dual codimension 1 face. The $\TT$-invariant prime Weil divisors of horizontal type on $\tX$ corresponding to $\rho$ is not contracted by $\varphi$ if and only if $\DD(m)$ is big for all $m\in\relint(\tau)$.
\item Let $D$ be a prime Weil divisor on $Y$ and let $p$ be a vertex of $\Delta_D$. The $\TT$-invariant prime Weil divisors on $\tX$ of vertical type $(D,p)$ is not contracted by $\varphi$ if and only if $\DD(m)|_D$ is big for all $m\in\relint((\Delta_D-p)^\vee)$. 
\end{enumerate}
\end{lemma}

\begin{corollary} \label{Tiso1}
The morphism $\varphi:\tX\rightarrow X$ is an isomorphism in codimension 1 if and only if the following conditions hold.
\begin{enumerate}[$(i)$]
\item For every codimension 1 face $\tau\subseteq\sigma^\vee$, the divisor $\DD(m)$ is big for all $m\in\relint(\tau)$.
\item For every prime Weil divisor $D$ on $Y$ and every vertex on $\Delta_D$, the divisor $\DD(m)|_Z$ is big for all $m\in\relint((\Delta_D-p)^\vee)$. 
\end{enumerate}
\end{corollary}
\begin{proof}
To prove that the equivariant contraction $\varphi:\tX\rightarrow X$ is an isomorphism in codimension 1 we only need to prove that no $\TT$-invariant Weil divisor is contracted by $\varphi$. The first condition ensures that no divisor of horizontal type is contracted and the second condition ensures that no divisor of vertical type is contracted.
\end{proof}

\begin{remark}
In the case of a complexity 1 $\TT$-action i.e., when $Y$ is a smooth curve, the condition $(ii)$ in Lemma \ref{Tcont} and Corollary \ref{Tiso1} is trivially verified.
\end{remark}

For one of our applications we need the following lemma concerning the Picard group of $\TT$-varieties, see Proposition 3.1 in \cite{PeSu08} for a particular case.

\begin{lemma} \label{spic}
Let $X=\spec A[Y,\DD]$, where $\DD$ is a proper $\sigma$-polyhedral divisor on a normal semiprojective variety $Y$. If $Y$ is projective then $\pic(X)$ is trivial.
\end{lemma}

\begin{proof}
Let $D$ be a Cartier divisor on $X$, and let $f$ be a local equation of $D$ in an open set $U\subseteq X$ containing $\bar{0}$. By \cite[\S 1, Exercise 16]{Bou65} we may assume that $D$ and $U$ are $\TT$-invariant. Since $\bar{0}$ is an attractive fixed point, every $\TT$-orbit closure contains $\bar{0}$ and so $U=X$, proving the lemma. 
\end{proof}

\subsection{Smooth polyhedral divisors}

The combinatorial description in Theorem \ref{AH} is not unique. The following Lemma is a specialization of Corollary 8.12 in \cite{AlHa06}. For the convenience of the reader, we provide a short argument.

\begin{lemma}\label{proj}
Let $\DD$ be a proper $\sigma$-polyhedral divisor on a normal semiprojective variety $Y$. Then for any projective birational morphism $\psi:Z\rightarrow Y$ the variety $\spec A[Y,\DD]$ is equivariantly isomorphic to $\spec A[Z,\psi^*\DD]$.
\end{lemma}

\begin{proof} 
We only need to show that 
$$H^0(Y,\OO_Y(\DD(m)))\simeq H^0(Z,\OO_Z(\psi^*\DD(m))),\mbox{ for all }m\in\sigma^\vee_M\,.$$
Letting $r$ be such that $r\DD(m)$ is Cartier $\forall m\in\sigma_M^\vee$, we have
$$H^0(Y,\OO_Y(\DD(m)))=\{f\in k(Y): f^r\in H^0(Y,r\DD(m))\},\quad \forall m\in\sigma_M^\vee\,.$$
Since $Y$ is normal and $\psi$ is projective, by Zariski main theorem $\psi_*\OO_Z=\OO_Y$ and by the projection formula, for all $m\in\sigma_M^\vee$ we have
\begin{align*}
H^0(Y,\OO_Y(\DD(m)))&\simeq\big\{f\in k(Z): f^r\in H^0(Z,\OO_Z(\psi^*r\DD(m)))\big\} \\
                    &=H^0(Z,\OO_Z(\psi^*\DD(m)))\,.
\end{align*}
This completes the proof.
\end{proof}

\begin{remark}
In the previous Lemma, $\tX=\spec\tA[Y,\DD]$ is not equivariantly isomorphic to $\spec\tA[Z,\psi^*\DD]$, unless $\psi$ is an isomorphism.
\end{remark}

To restrict further the class of $\sigma$-polyhedral divisor we introduce the following notation.

\begin{definition}
We say that a $\sigma$-polyhedral divisor $\DD$ on a normal semiprojective variety $Y$ is smooth if $Y$ is smooth, $\DD$ is proper, and $\supp\DD$ is SNC.
\end{definition}

\begin{remark} \label{smooth1}
In the case of complexity one i.e., when $Y$ is a curve, any proper $\sigma$-polyhedral divisor is smooth. Indeed, any normal curve is smooth and any divisor on a smooth curve is SNC.
\end{remark}

\begin{corollary}
For any  $\TT$-variety $X$ there exists  a smooth $\sigma$-polyhedral divisor on a smooth semiprojective variety $Y$ such that $X=\spec A[Y,\DD]$.
\end{corollary}
\begin{proof}
By Theorem \ref{AH}, there exists a proper $\sigma$-polyhedral divisor $\DD'$ on a normal semiprojective variety $Y'$ such that $X=\spec A[Y,\DD]$. Let $\psi:Y\rightarrow Y'$ be a resolution of singularities of $Y$ such that $\supp\DD'$ is SNC. By Chow Lemma we can assume that $Y$ is semiprojective. By Lemma \ref{proj}, $\DD=\psi^*\DD'$ is a smooth $\sigma$-polyhedral divisor such that $X=\spec A[Y,\DD]$.
\end{proof}

In the sequel, unless the converse is explicitly stated, we restrict to smooth $\sigma$-polyhedral divisors.


\section{Singularities of $\TT$-varieties}

In this section we elaborate a method to effectively compute an equivariant partial resolution of singularities of an affine $\TT$-variety in terms of the combinatorial data $(Y,\DD)$. As a consequence, we compute the higher direct image of the structure sheaf for any resolution of singularities, and we give a criterion as to when a $\TT$-variety has rational singularities.

A key ingredient for our results is the following example (cf. Example 3.19 in \cite{Lie08}).

\begin{example} \label{extor}
Let $H_i$, $i\in\{1,\ldots,n\}$ be the coordinate hyperplanes in $Y=\AF^n$, and let a smooth divisor $\DD$ on $Y$ given by 
$$\DD=\sum_{i=0}^{n}\Delta_i\cdot H_i,\quad \mbox{where } \Delta_i\in\pol_\sigma(N_{\QQ})\,.$$
Letting $h_i=h_{\Delta_i}$ be the support function of $\Delta_i$ and $k(Y)=k(t_1,\ldots, t_n)$, we have
\begin{align*}
H^0(Y,\OO_Y(\DD(m)))&=\big\{f\in k(Y):\divi(f)+\DD(m)\geq 0\big\} \\
                    &=\left\{f\in k(Y):\divi(f)+\sum_{i=1}^n h_i(m)\cdot H_i\geq 0\right\} \\
                    &=\bigoplus_{r_i\geq -h_i(m)}k\cdot t_1^{r_1}\cdots t_n^{r_n}\,.
\end{align*}
Let $\widehat{N}=N \times \ZZ^n$, $\widehat{M}=M\times \ZZ^n$ and $\widehat{\sigma}$ be the cone in $\widehat{N}_{\QQ}$ spanned by $(\sigma,\overline{0})$ and $(\Delta_i,e_i)$, $\forall i\in\{1,\ldots,n\}$, where $e_i$ is the $i$-th vector in the standard base of $\QQ^n$. A vector $(m,r)\in M'$ belongs to the dual cone $(\sigma')^\vee$ if and only if $m\in\sigma^\vee$ and $r_i\geq -h_i(m)$.

With this definitions we have
$$A[Y,\DD]=\bigoplus_{m\in\sigma_M^\vee} H^0(Y,\OO_Y(\DD(m)))=\bigoplus_{(m,r) \in\widehat{\sigma}_{\widehat{M}}^\vee}k\cdot t_1^{r_1}\cdots t_n^{r_n}\simeq k[\widehat{\sigma}^\vee_{\widehat{M}}]\,.$$

Hence $X=\spec A[Y,\DD]$ is isomorphic as an abstract variety to the toric variety with cone $\widehat{\sigma}\subseteq \widehat{N}_{\QQ}$. Since $Y$ is affine $\tX\simeq X$ is also a toric variety.
\end{example}

\begin{definition}
Recall \cite{KKMS73} that a variety $X$ is toroidal if for every $x\in X$ there is a formal neighborhood isomorphic to a formal neighborhood of a point $y\in U_{\omega}$ in a toric variety.
\end{definition}

\begin{lemma} \label{toroidal}
Let $\DD=\sum_D \Delta_D\cdot D$ be a proper $\sigma$-polyhedral divisor on a semiprojective normal variety $Y$. If $\DD$ is smooth then $\tX=\spec_Y \tA[Y,\DD]$ is a toroidal variety.
\end{lemma}

\begin{proof}
For $y\in Y$ we consider the fiber $X_y$ over $y$ for the morphism $\varphi:\tX\rightarrow Y$. We let also $\mathfrak{U}_y$ be a formal neighborhood of $X_y$. 

We let $S_y=\{D\in\operatorname{Wdiv}(Y), y\in D\mbox{ and } \Delta_D\neq\sigma\}$ and $n=\dim Y$. Since $\supp\DD$ is SNC, we have that $\operatorname{card}(S_y)\leq n$. We let $j:S_y\rightarrow \{1,\ldots,n\}$ be any injective function.

We consider the smooth $\sigma$-polyhedral divisor
$$\DD_y'= \sum_{D\in S_y} \Delta_D\cdot H_{j(D)}, \quad \mbox{on} \quad \AF^n\,.$$

Since $Y$ is smooth, $\mathfrak{U}_y$ is isomorphic to a formal neighborhood of the fiber over zero for the morphism $\pi':\spec \tA[\AF^n,\DD_y']\rightarrow \AF^n$ (see Theorem \ref{AH} $(ii)$).

Finally, Example \ref{extor} shows that $\spec \tA[\AF^n,\DD_y']$ is toric for all $y$ and so $X$ is toroidal. This completes the proof.
\end{proof}

\begin{remark}
Since the contraction $\varphi:\tX\rightarrow X=\spec A[Y,\DD]$ in Theorem \ref{AH} is proper and birational, to obtain a desingularization of $X$ it is enough to have a desingularization of $\tX$. If further $\DD$ is smooth, then $\tX$ is toroidal and there exists a toric desingularization.
\end{remark}

\subsection{Rational Singularities}

In the following we use the partial desingularization $\varphi:\tX\rightarrow X$ where $\tX$ is toroidal as in Theorem \ref{AH}. This allows us to provide information about the singularities of $X$ in terms of the combinatorial data $(Y,\DD)$. We recall the following notion.

\begin{definition}
A variety $X$ has rational singularities if there exists a desingularization $\psi:Z\rightarrow X$, such that
$$\psi_*\OO_Z=\OO_X,\quad\mbox{and}\quad R^i\psi_*\OO_Z=0, \quad \forall i>0\,.$$
\end{definition}

\begin{remark} \label{anydes}
This definition is correct since the higher direct image sheaves $R^i\psi_*\OO_Z$ are independent of the particular choice of the desingularization of $X$. The first condition $\psi_*\OO_Z=\OO_X$ is equivalent to $X$ is normal.
\end{remark}

The following well known Lemma, similarly as in \cite{Vieh77}, follows by applying the Leray spectral sequence. For the convenience of the reader we provide a short argument.
\begin{lemma} \label{rati}
Let $\varphi:\tX\rightarrow X$ be a proper surjective, birational morphism, and let $\psi:Z\rightarrow X$ be a desingularization of $X$. If $\tX$ has only rational singularities, then
$$R^i\psi_*\OO_Z=R^i\varphi_*\OO_{\tX}, \quad \forall i\geq0\,.$$
\end{lemma}
\begin{proof}
By Remark \ref{anydes}, we may assume that the desingularization $\psi$ is such that $\psi=\varphi\circ\widetilde{\psi}$, where $\widetilde{\psi}:Z\rightarrow\tX$ is a desingularization of $\tX$. The question is local on $X$, so we may assume that $X$ is affine. Then, by \cite[Ch. III, Prop. 8.5]{Har77} we have\footnote{As usual for a $A$-module $M$, $M^{\sim}$ denontes the associated sheaf on $X=\spec A$.}
$$R^i\psi_*\OO_Z=H^i(Z,\OO_Z)^{\sim}\quad\mbox{and}\quad R^i\varphi_*\OO_{\tX}=H^i(\tX,\OO_{\tX})^{\sim},\quad \forall i\geq 0\,.$$
Since $\tX$ has rational singularities
$$\widetilde{\psi}_*\OO_Z=\OO_{\tX},\quad\mbox{and}\quad R^i\widetilde{\psi}_*\OO_Z=0, \quad \forall i>0\,.$$
By Leray spectral sequence for $(p,q)=(i,0)$ we have
$$H^i(Z,\OO_Z)=H^i(\tX,\widetilde{\psi}_*\OO_Z)=H^i(\tX,\OO_{\tX}),\quad \forall i\geq 0\,,$$
proving the Lemma.
\end{proof}

In the following Theorem, which is our main result, for a $\TT$-variety $X=\spec A[Y,\DD]$ and a desingularization $\psi:Z\rightarrow X$ we provide an expression for $R^i\psi_*\OO_Z$ in terms of the combinatorial data $(Y,\DD)$.

\begin{theorem} \label{Tdirect}
Let $X=\spec A[Y,\DD]$, where $\DD$ is a smooth polyhedral divisor on $Y$. If $\psi:Z\rightarrow X$ is a desingularization, then $R^i\psi_*\OO_Z$ is the sheaf associated to
$$\bigoplus_{m\in\sigma_M^\vee} H^i(Y,\OO_Y(\DD(m)))$$
\end{theorem}

\begin{proof}
Let $\varphi:\tX\rightarrow X$ be as in Theorem \ref{AH}. By Lemma \ref{toroidal} $\tX$ is toroidal, thus it has only rational singularities. By Lemma \ref{rati} we have
$$R^i\psi_*\OO_Z=R^i\varphi_*\OO_{\tX}, \quad \forall i\geq0\,.$$

Since $X$ is affine,  we have
$$R^i\varphi_*\OO_{\tX}=H^i(\tX,\OO_{\tX})^{\sim}, \quad \forall i\geq0\,,$$
see \cite[Ch. III, Prop. 8.5]{Har77}. Let $\pi:\tX\rightarrow Y$ be the good quotient in Theorem \ref{AH} and let $\tA=\bigoplus_{m\in\sigma_M^\vee} \OO_Y(\DD(m))$ so that $\tX=\spec_Y\tA$. Since the morphism $\pi$ is affine, we have
$$H^i(\tX,\OO_{\tX})= H^i(Y,\tA)= \bigoplus_{m\in\sigma_M^\vee} H^i(Y,\OO_Y(\DD(m))),\quad \forall i\geq0$$
by \cite[Ch III, Ex. 4.1]{Har77}, proving the Theorem.
\end{proof}

As an immediate consequence of Theorem \ref{Tdirect}, in the following theorem, we characterize $\TT$-varieties having rational singularities.

\begin{theorem} \label{Trat}
Let $X=A[Y,\DD]$, where $\DD$ is a smooth $\sigma$-polyhedral divisor on $Y$. Then $X$ has rational singularities if and only if for every $m\in\sigma_M^\vee$
$$H^i(Y,\OO_Y(\DD(m)))=0, \quad \forall i\in \{1,\ldots,\dim Y\}\,.$$
\end{theorem}

\begin{proof}
Since $X$ is normal, by Theorem \ref{Tdirect} we only have to prove that
$$\bigoplus_{m\in\sigma_M^\vee} H^i(Y,\OO_Y(\DD(m)))=0,\quad \forall i>0$$
This direct sum is trivial if and only if each summand is. Hence $X$ has rational singularities if and only if $H^i(Y,\OO_Y(\DD(m)))=0$, for all $i>0$ and all $m\in\sigma_M^\vee$.

Finally, $H^i(Y,\FF)=0$, for all $i>\dim Y$ and for any sheaf $\FF$, see \cite[Ch III, Th. 2.7]{Har77}. Now the Lemma follows.
\end{proof}

In particular, we have the following Corollary.

\begin{corollary} \label{acyc}
Let $X=A[Y,\DD]$ for some smooth $\sigma$-polyhedral divisor $\DD$ on $Y$. If $X$ has only rational singularities, then the structure sheaf $\OO_Y$ is acyclic i.e., $H^i(Y,\OO_Y)=0$ for all $i>0$.
\end{corollary}
\begin{proof}
This is the ``only if'' part of Theorem \ref{Trat} for $m=0$.
\end{proof}

In the case of complexity 1 i.e., when $Y$ is a curve, there is a more explicit criterion.

\begin{corollary} \label{rcom1}
Let $X=A[Y,\DD]$, where $\DD$ is a smooth $\sigma$-polyhedral divisor on a smooth curve $Y$. Then $X$ has only rational singularities if and only if 
\begin{enumerate}[$(i)$]
 \item $Y$ is affine, or 
 \item $Y=\PP^1$ and $\deg\lfloor\DD(m)\rfloor \geq -1$ for all $m\in\sigma_M^\vee$.
\end{enumerate}
\end{corollary}

\begin{proof}
If $Y$ is affine then by Remark \ref{aff}, the morphism $\varphi$ in Theorem \ref{AH} is an isomorphism. By Lemma \ref{toroidal} $X$ is toroidal and thus $X$ has only rational singularities.

If $Y$ is projective of genus $g$, we have $\dim H^1(Y,\OO_Y)=g$. So by Corollary \ref{acyc} if $X$ has rational singularities then $C=\PP^1$. Furthermore, for the projective line we have $H^1(\PP^1,\OO_{\PP^1}(D))\neq 0$ if and only if $\deg D\leq-2$ \cite[Ch. III, Th 5.1]{Har77}. Now the corollary follows from Theorem \ref{Trat}.
\end{proof}

\begin{remark}
Let $\DD$ be a smooth divisor on $Y$. Assume that $X=\spec A[Y,\DD]$ has log-terminal (or canonical or terminal) singularities. Then $X$ has rational singularities and so $(Y,\DD)$ satisfies the assumptions of Theorem \ref{Trat}. In particular, if $Y$ is a projective curve then $Y$ must be rational. This complements Proposition 3.9 in \cite{Sus08} by showing that the case with $Y$ an elliptic curve cannot happen.
\end{remark}


\subsection{Cohen-Macaulay singularities}

In the following, we apply Theorem \ref{Trat} to give a partial classification of Cohen-Macaulay $\TT$-varieties in terms of the combinatorial description $(Y,\DD)$.

Recall that a local ring is Cohen-Macaulay if its Krull dimension is equal to its depth. A variety is Cohen-Macaulay if all its local rings are. The following lemma is well known, see for instance \cite[page 50]{KKMS73}.

\begin{lemma} \label{kempf}
Let $\psi:Z\rightarrow X$ be a desingularization of $X$. Then $X$ has rational singularities if and only if $X$ is Cohen-Macaulay and $\psi_*\omega_Z\simeq\omega_X$ \footnote{As usual $\omega_Z$ and $\omega_X$ denote the canonical sheaf of $Z$ and $X$ respectively.}.
\end{lemma}

\begin{remark} \label{rm-cm}
\begin{enumerate}[$(i)$]
 \item As in Lemma \ref{rati}, applying the Leray spectral sequence the previous lemma is still valid if we allow $Z$ to have rational singularities.
 \item If $\psi$ is an isomorphism in codimension 1 then $\psi_*\omega_Z\simeq\omega_X$.
\end{enumerate}
\end{remark}

In the next corollary, we give a partial criterion as to when a $\TT$-variety has only Cohen-Macaulay singularities.

\begin{corollary} \label{cmgen}
Let $X=A[Y,\DD]$, where $\DD=\sum_D\Delta_D\cdot D$ is a smooth $\sigma$-polyhedral divisor on $Y$. Assume that the following conditions hold.
\begin{enumerate}[$(i)$]
\item For every codimension 1 face $\tau\subseteq\sigma^\vee$, the divisor $\DD(m)$ is big for all $m\in\relint(\tau)$.
\item For every prime Weil divisor $D$ on $Y$ and every vertex on $\Delta_D$, the divisor $\DD(m)|_Z$ is big for all $m\in\relint((\Delta_D-p)^\vee)$. 
\end{enumerate}
Then $X$ is Cohen-Macaulay if and only if $X$ has rational singularities.
\end{corollary}
\begin{proof}
By Corollary \ref{Tiso1}, the contraction $\varphi:\tX\rightarrow X$ is an isomorphism in codimension 1. The result now follows from Lemma \ref{kempf} and Remark \ref{rm-cm} (ii).
\end{proof}

In the next corollary, we provide more explicit criterion for the case of complexity 1.

\begin{corollary} \label{cm1}
Let $X=A[Y,\DD]$, where $Y$ is a smooth curve and $\DD$ is a smooth $\sigma$-polyhedral divisor on $Y$. If one of the following conditions hold,
\begin{enumerate}[$(i)$]
 \item $Y$ is affine, or
 \item $\rank M=1$, or
 \item $Y$ is projective and $\deg\DD\cap\rho=\emptyset$, for every extremal ray $\rho\subseteq \sigma$.
\end{enumerate}
Then $X$ is Cohen-Macaulay if and only if $X$ has rational singularities.
\end{corollary}

\begin{proof}
If $Y$ is affine then $X$ has rational singularities by Corollary \ref{rcom1}, and so $X$ is Cohen-Macaulay. If $\rank M=1$ then $X$ is a normal surface. By Serre S$_2$ normality criterion any normal surface is Cohen-Macaulay, see Theorem 11.5 in \cite{Eis95}. Finally, if $(iii)$ holds then by Lemma \ref{Tiso1},  $\varphi:\tX\rightarrow X$ is an isomorphism in codimension 1. Now the result follows by Lemma \ref{kempf}.
\end{proof}

For isolated singularities we can give a full classification whenever $\rank M\geq 2$.

\begin{corollary}
Let $X=A[Y,\DD]$, where $D$ is a smooth $\sigma$-polyhedral divisor on $Y$. If $\rank M\geq 2$ and $X$ has only isolated singularities, then $X$ is Cohen-Macaulay if and only if $X$ has rational singularities.
\end{corollary}
\begin{proof}
The only thing we have to prove is the ``only if'' part. Assume that $X$ is Cohen-Macaulay and let $\psi:Z\rightarrow X$ be a resolution of singularities. Since $X$ has only isolated singularities we have that $R^i\psi_*O_Z$ vanishes except possibly for $i=\dim X-1$, see \cite[Lemma 3.3]{Kov99}. Now Theorem \ref{Tdirect} shows that $R^i\psi_*O_Z$ vanishes also for $i=\dim X-1$ since $\dim Y=\dim X -\rank M$ and $\rank M\geq 2$.
\end{proof}

\begin{remark}
The last two corollaries give a full classification of isolated Cohen-Macaulay singularities on $\TT$-varieties of complexity 1.
\end{remark}



\section{Quasihomogeneous surfaces singularities}

In this section we study in more detail the particular case of a one dimensional torus action of complexity one i.e., the case of $\KK^*$-surfaces. We characterize Gorenstein and elliptic singularities in terms of the combinatorial data as in Theorem \ref{AH}.

Let $X=\spec A[Y,\DD]$ be a $\KK^*$-surface, so that $Y$ is a smooth curve and $M\simeq \ZZ$. There are only two non-equivalent pointed polyhedral cones in $N_\QQ\simeq\QQ$ corresponding to $\sigma=\{0\}$ and $\sigma=\QQ_{\geq 0}$, and any $\sigma$-polyhedral divisor $\DD$ on $Y$ is smooth.

With the notation of Theorem \ref{AH} suppose that $Y$ is affine. Then $X\simeq\tX$ by Remark \ref{aff} and so $X$ is toroidal by Lemma \ref{toroidal}. In this case the singularities of $X$ can be classified by toric methods. In particular they are all rational, see for instance \cite{Oda88}.

If $Y$ is projective, then $\sigma\neq \{0\}$ and so we can assume that $\sigma=\QQ_{\geq 0}$. In this case $\DD(m)=m\DD(1)$. Hence $\DD$ is completely determined by $\DD_1:=\DD(1)$.

Furthermore,
$$A[Y,\DD]=\bigoplus_{m\geq0}A_m\chi^m,\quad \mbox{where} \quad A_m=H^0(Y,m\DD_1)\,.$$
and there is an unique atractive fixed point $\bar{0}$ corresponding to the augmentation ideal $\mathfrak{m}_0=\bigoplus_{m>0}A_m\chi^m$.

This is exactly the setting studied in \cite{FlZa03c}, where all $\KK^*$-surfaces are divided in three types: elliptic, parabolic and hyperbolic. In combinatorial language these correspond, respectively,  to the cases where $Y$ is projective and $\sigma=\QQ_{\geq 0}$, $Y$ is affine and $\sigma=\QQ_{\geq 0}$, and finally $Y$ is affine and $\sigma=\{0\}$.

In particular, in \cite{FlZa03c} invariant divisors on $\KK^*$-surfaces are studied. The results in \emph{loc.cit.} are stated only for the hyperbolic case. However, similar statements for the remaining cases can be obtained with essentially the same proofs. In the recent preprint \cite{Sus08} some of the results in \emph{loc.cit.} have been  generalized to the case of $\rank M>1$. Let us recall the necessary results from \cite[\S 4]{FlZa03c}, see also \cite{Sus08}.

Let $X=\spec A[Y,\DD]$, where $\DD$ is a smooth $\sigma$-polyhedral divisor on be a projective smooth curve $Y$ (the elliptic case), and let as before $\DD_1=\DD(1)$. We can write
$$\DD_1=\sum_{i=1}^{\ell} \frac{p_i}{q_i}z_i,\quad\mbox{where}\quad \gcd(p_i,q_i)=1,\mbox{ and }q_i>0\,.$$

In this case, with the notation of Theorem \ref{AH} the birational morphism $\rho:=\pi\circ\varphi^{-1}:X\rightarrow Y$ is surjective and its indeterminacy locus consists of the unique fixed point corresponding to the maximal ideal $\bigoplus_{m>0}A_i$. The $\KK^*$-invariant prime divisors are $D_{z}:=\rho^{-1}(z)$, $\forall z\in Y$. The total transforms are: $\rho^*(z)=D_z$ for all $z\notin\supp\DD_1$, and $\rho^*(z_i)=q_iD_{z_i}$, for $i=1,\ldots,\ell$. We let $D_i=D_{z_i}$ for  $i=1,\ldots,\ell$. 

The canonical divisor of $X$ is given by
$$K_X=\rho^*(K_Y)+\sum_{i=1}^\ell (q_i-1)D_i\,.$$
For a rational semi-invariant function $f\cdot\chi^m$, where $f\in K(Y)$ and $m\in\ZZ$, we have
$$\divi(f\cdot\chi^m)=\rho^*(\divi f)+m\sum_{i=1}^{\ell} p_iD_i\,.$$

For our next result we need the following notation.

\begin{notation} \label{nog}
We let
\begin{align} \label{nog-1}
m_G=\frac{1}{\deg\DD_1}\left(\deg K_Y+ \sum_{i=1}^{\ell}\frac{q_i-1}{q_i} \right)\,,
\end{align}
and
\begin{align} \label{nog-2}
D_G=\sum_{i=1}^{\ell}d_iz_i,\quad\mbox{where}\quad d_i=\frac{p_im_G+1}{q_i}-1,\quad \forall i\in\{1,\ldots,\ell\}\,.
\end{align}
\end{notation}

Recall that a variety $X$ is \emph{Gorenstein} if it is Cohen-Macaulay and the canonical divisor $K_X$ is Cartier. By Serre S$_2$ normality criterion, all normal surface singularities are Cohen-Macaulay. In the following proposition we give a criterion for a $\KK^*$-surface to have Gorenstein singularities.

\begin{proposition} \label{gor}
Let $X=\spec A[Y,\DD]$, where $\DD$ is a smooth $\sigma$-polyhedral divisor on a smooth projective curve $Y$. With the notation as in \ref{nog}, the surface $X$ has Gorenstein singularities if and only if $m_G$ is integral and $D_G-K_Y$ is a principal divisor on $Y$.
\end{proposition}

\begin{proof}
By Lemma \ref{spic}, $X$ is Gorenstein if and only if $K_X$ is a principal divisor i.e., there exist $m_G\in\ZZ$ and a principal divisor $D=\divi(f)$ on $Y$  such that
$$K_X=\rho^*(K_Y)+\sum_{i=1}^\ell (q_i-1)D_i=\rho^*D+m_G\sum_{i=1}^{\ell}p_iD_i=\divi(f\cdot\chi^{m_G}),.$$
Clearly $\supp(K_Y-D)\subseteq\{z_1\ldots,z_\ell\}$. Letting 
$$K_Y-D=\sum_{i=1}^{\ell}d_iz_i$$
we obtain
$$\sum_{i=1}^\ell q_id_iD_i=\sum_{i=1}^\ell (mp_i-q_i+1)D_i\,.$$
Hence the $d_i$ satisfy \eqref{nog-1} in \ref{nog}. Furthermore, since 
$$\deg K_Y=\deg(K_Y-D)=\sum_{i=1}^\ell d_i\,,$$
$m_G$ satisfies \eqref{nog-2} in \ref{nog}. So $X$ is Gorenstein if and only if $m_G$ is integral and $D=K_Y-D_G$ is principal, proving the proposition.
\end{proof}

Let $(X,x)$ be a normal surface singularity, and let $\psi:Z\rightarrow X$ be a resolution of the singularity $(X,x)$. One says that $(X,x)$ is an \emph{elliptic singularity}\footnote{Some authors call such $(X,x)$ a strongly elliptic singularity.} if $R^1\psi_*\OO_Z\simeq \KK$. An elliptic singularity is \emph{minimal} if it is Gorenstein. e.g., \cite{Lau77}, \cite{Wat80}, and \cite{Yau80}.

In the following theorem we characterize quasihomogeneous (minimal) elliptic singularities of surfaces.
\begin{theorem} \label{ellip}
Let $X=\spec A[Y,\DD]$ be a normal affine surface with an effective elliptic 1-torus action, and let $\bar{0}\in X$ be the unique fixed point. Then $(X,\bar{0})$ is an elliptic singularity if and only if one of the following two conditions holds:
\begin{enumerate}[$(i)$]
 \item $Y=\PP^1$, $\deg\lfloor m\DD_1\rfloor\geq-2$ and $\deg\lfloor m\DD_1\rfloor=-2$ for one and only one $m\in\ZZ_{>0}$.
 \item $Y$ is an elliptic curve, and for every $m\in\ZZ_{>0}$, the divisor $\lfloor m\DD_1\rfloor$ is not principal and $\deg\lfloor m\DD_1\rfloor\geq 0$.
\end{enumerate}
Moreover, $(X,\bar{0})$ is a minimal elliptic singularity if and only if $(i)$ or $(ii)$ holds, $m_G$ is integral and $D_G-K_Y$ is a principal divisor on $Y$, where $m_G$ and $D_G$ are as in \ref{nog}. 
\end{theorem}

\begin{proof}
Assume that $Y$ is a projective curve of genus $g$, and let $\psi:Z\rightarrow X$ be a resolution of singularities. By Theorem \ref{Tdirect},
$$R^1\psi_*\OO_Z=\bigoplus_{m\geq0}H^1(Y,\OO_Y(m\DD_1))\,.$$
Since $\dim R^1\psi_*\OO_Z\geq g=\dim H^1(Y,\OO_Y)$, if $X$ has an elliptic singularity then $g\in\{0,1\}$. 

If $Y=\PP^1$ then $(X,\bar{0})$ is an elliptic singularity if and only if $H^1(Y,\OO_Y(m\DD_1)=\KK$ for one and only one value of $m$. This is the case if and only if $(i)$ holds.

If $Y$ is an elliptic curve, then $H^1(Y,\OO_Y)=\KK$. So the singularity $(X,\bar{0})$ is elliptic if and only if $H^1(Y,m\DD_1)=0$ for all $m>0$. This is the case if and only if $(ii)$ holds.

Finally, the last assertion concerning maximal elliptic singularities follows immediately form Proposition \ref{gor}.
\end{proof}

\begin{example}
By applying the criterion of Theorem \ref{ellip}, the following combinatorial data gives rational $\KK^*$-surfaces with an elliptic singularity at the only fixed point.
\begin{enumerate}[$(i)$]
 \item $Y=\PP^1$ and $\DD_1=-\tfrac{1}{4}[0]-\tfrac{1}{4}[1]+\tfrac{3}{4}[\infty]$. In this case $X=\spec A[Y,m\DD_1]$ is isomorphic to the surface in $\AF^3$ with equation
$$x_1^4x_3+x_2^3+x_3^2=0\,.$$
 \item $Y=\PP^1$ and $\DD_1=-\tfrac{1}{3}[0]-\tfrac{1}{3}[1]+\tfrac{3}{4}[\infty]$. In this case $X=\spec A[Y,m\DD_1]$ is isomorphic to the surface in $\AF^3$ with equation
$$x_1^4+x_2^3+x_3^3=0\,.$$
 \item $Y=\PP^1$ and $\DD_1=-\tfrac{2}{3}[0]-\tfrac{2}{3}[1]+\tfrac{17}{12}[\infty]$. In this case $X=\spec A[Y,m\DD_1]$ is isomorphic to the surface
$$V(x_1^4x_2x_3-x_2x_3^2+x_4^2\ ;\, x_1^5x_3-x_1x_3^2+x_2x_4\ ;\,x_2^2-x_1x_4)\subseteq \AF^4\,.$$
\end{enumerate}
This last example is not a complete intersection since otherwise $(X,\bar{0})$ would be Gorenstein i.e., minimal elliptic which is not the case by virtue of Theorem \ref{ellip}. In the first two examples the elliptic singlarities are minimal, since every normal hypersurface is Gorenstein.
\end{example}


\bibliographystyle{alphanum}
\bibliography{math}

\end{document}